\documentclass[11pt]{article}

\usepackage{amssymb,amsmath,amsfonts,amsthm}
\usepackage{latexsym}
\usepackage{graphics}
\usepackage{indentfirst}

\newtheorem{innercustomthm}{Theorem}
\newenvironment{customthm}[1]
  {\renewcommand\theinnercustomthm{#1}\innercustomthm}
  {\endinnercustomthm}

\newtheorem{innercustomcor}{Corollary}
\newenvironment{customcor}[1]
  {\renewcommand\theinnercustomcor{#1}\innercustomcor}
  {\endinnercustomcor}
\newtheorem*{thm*}{Theorem}
\newtheorem{thm}{Theorem}
\newtheorem{lem}[thm]{Lemma}
\newtheorem{pro}[thm]{Proposition}

\newtheorem{cor}[thm]{Corollary}

\newtheorem{ques}[thm]{Question}

\newcommand{\N}{\mathbb{N}}

\newcommand{\col}{\mathrm{col}}

\begin{document}

\title{List Coloring a Cartesian Product with a Complete Bipartite Factor}

\author{Hemanshu Kaul\footnote{Department of Applied Mathematics, Illinois Institute of Technology, Chicago, IL 60616. E-mail: {\tt kaul@iit.edu}} \\
Jeffrey A. Mudrock\footnote{Department of Mathematics, College of Lake County, Grayslake, IL 60030. E-mail: {\tt jmudrock@clcillinois.edu}} }

\maketitle

\begin{abstract}

We study the list chromatic number of the Cartesian product of any graph $G$ and a complete bipartite graph with partite sets of size $a$ and $b$, denoted $\chi_\ell(G \square K_{a,b})$. We have two motivations. A classic result on the gap between list chromatic number and the chromatic number tells us $\chi_\ell(K_{a,b}) = 1 + a$ if and only if $b \geq a^a$.  Since $\chi_\ell(K_{a,b}) \leq 1 + a$ for any $b \in \N$, this result tells us the values of $b$ for which $\chi_\ell(K_{a,b})$ is as large as possible and far from $\chi(K_{a,b})=2$. In this paper we seek to understand when $\chi_\ell(G \square K_{a,b})$ is far from $\chi(G \square K_{a,b}) = \max \{\chi(G), 2 \}$.  It is easy to show $\chi_\ell(G \square K_{a,b}) \leq \chi_\ell (G) + a$. In 2006, Borowiecki, Jendrol, Kr{\'a}l, and Mi{\v s}kuf showed that this bound is attainable if $b$ is sufficiently large; specifically, $\chi_\ell(G \square K_{a,b}) = \chi_\ell (G) + a$ whenever $b \geq (\chi_\ell(G) + a - 1)^{a|V(G)|}$.  Given any graph $G$ and $a \in \N$, we wish to determine the smallest $b$ such that $\chi_\ell(G \square K_{a,b}) = \chi_\ell (G) + a$. In this paper we show that the list color function, a list analogue of the chromatic polynomial, provides the right concept and tool for making progress on this problem.  Using the list color function, we prove a general improvement on Borowiecki et al.'s 2006 result, and we compute the smallest such $b$ exactly for some large families of chromatic-choosable graphs.

\medskip



\end{abstract}

\section{Introduction}\label{intro}

In this paper all graphs are nonempty, finite, simple graphs unless otherwise noted.  Generally speaking we follow West~\cite{W01} for terminology and notation.  The set of natural numbers is $\N = \{1,2,3, \ldots \}$.  For $k \in \N$, we write $[k]$ for the set $\{1,2, \ldots, k \}$.  If $G$ is a graph and $S \subseteq V(G)$, we write $G[S]$ for the subgraph of $G$ induced by $S$.  For $v \in V(G)$, we write $d_G(v)$ for the degree of vertex $v$ in the graph $G$.  If $G$ and $H$ are vertex disjoint graphs, we write $G \vee H$ for the join of $G$ and $H$.

\subsection{List Coloring}

List coloring is a variation on the classic vertex coloring problem.  It was introduced in the 1970's independently by Vizing~\cite{V76} and Erd\H{o}s, Rubin, and Taylor~\cite{ET79}.  In the classic vertex coloring problem we seek a \emph{proper $k$-coloring} of a graph $G$ which is a coloring of the vertices of $G$ with colors from $[k]$ such that adjacent vertices receive different colors. The \emph{chromatic number} of a graph, denoted $\chi(G)$, is the smallest $k$ such that $G$ has a proper $k$-coloring.  For list coloring, we associate a \emph{list assignment}, $L$, with a graph $G$ such that each vertex $v \in V(G)$ is assigned a list of colors $L(v)$ (we say $L$ is a list assignment for $G$).  The graph $G$ is \emph{$L$-colorable} if there exists a proper coloring $f$ of $G$ such that $f(v) \in L(v)$ for each $v \in V(G)$ (we refer to $f$ as a \emph{proper $L$-coloring} of $G$).  A list assignment $L$ is called a \emph{k-assignment} for $G$ if $|L(v)|=k$ for each $v \in V(G)$.  The \emph{list chromatic number} of a graph $G$, denoted $\chi_\ell(G)$, is the smallest $k$ such that $G$ is $L$-colorable whenever $L$ is a $k$-assignment for $G$.  We say $G$ is \emph{$k$-choosable} if $k \geq \chi_\ell(G)$.

It is immediately obvious that for any graph $G$, $\chi(G) \leq \chi_\ell(G)$.  Both Vizing~\cite{V76} and Erd\H{o}s, Rubin, and Taylor~\cite{ET79} observed bipartite graphs can have arbitrarily large list chromatic number.  This implies the gap between $\chi(G)$ and $\chi_\ell(G)$ can be arbitrarily large.  The following result illustrates this.

\begin{thm} \label{thm: listbipartite}
For $a,b \in \N$, $\chi_\ell(K_{a,b})=a+1$ if and only if $b \geq a^a$.
\end{thm}

It is worth mentioning that Theorem~\ref{thm: listbipartite} is often attributed to Vizing~\cite{V76} or Erd\H{o}s, Rubin, and Taylor~\cite{ET79}, but it is not stated in those papers. It is best described as a folklore result.  In this paper we prove results similar in flavor to Theorem~\ref{thm: listbipartite} for Cartesian products.

Graphs in which $\chi(G) = \chi_\ell(G)$ are known as \emph{chromatic-choosable} graphs~\cite{O02}.  The notion of chromatic-choosability has received considerable attention in the literature.  Many families of graphs have been conjectured to be chromatic-choosable (see for example~\cite{BKW97}, \cite{GM97}, and~\cite{HC92}), and there are several families of graphs that have been shown to be chromatic-choosable (see for example~\cite{G95}, \cite{KM18}, \cite{NR15}, \cite{PW03}, and~\cite{TV96}).  We are studying how Cartesian products with a complete bipartite factor can be far from being chromatic-choosable.

\subsection{List Coloring Cartesian Products}

The \emph{Cartesian product} of graphs $M$ and $H$, denoted $M \square H$, is the graph with vertex set $V(M) \times V(H)$ and edges created so that $(u,v)$ is adjacent to $(u',v')$ if and only if either $u=u'$ and $vv' \in E(H)$ or $v=v'$ and $uu' \in E(M)$.  Throughout this paper, if $G = M \square H$ and $u \in V(M)$ (resp. $u \in V(H)$), we let $V_u$ be the subset of $V(G)$ consisting of the vertices with first (resp. second) coordinate $u$.  By the definition of Cartesian product of graphs, it is easy to see $G[V_u]$ is a copy of $H$ (resp. $M$), and we refer to $G[V_u]$ as the \emph{copy of $H$ (resp. $M$) corresponding to $u$}.

It is well known that $\chi(G \square H) = \max \{\chi(G), \chi(H) \}$.  On the other hand, the list chromatic number of the Cartesian product of graphs is not nearly as well understood.  In 2006, Borowiecki, Jendrol, Kr{\'a}l, and Mi{\v s}kuf~\cite{BJ06} showed the following.

\begin{thm}[\cite{BJ06}] \label{thm: Borow1}
For any graphs $G$ and $H$, $\chi_\ell(G \square H) \leq \min \{\chi_\ell(G) + \col(H), \col(G) + \chi_\ell(H) \} - 1.$
\end{thm}

Here $\col(G)$, the \emph{coloring number} of a graph $G$, is the smallest integer $d$ for which there exists an ordering, $v_1, v_2, \ldots, v_n$, of the elements in $V(G)$ such that each vertex $v_i$ has at most $d-1$ neighbors among $v_1, v_2, \ldots, v_{i-1}$.  For any graph $G$, it is easy to see that Theorem~\ref{thm: Borow1} implies $\chi_\ell(G \square K_{a,b}) \leq \chi_\ell(G) + a$.  In proving that the bound in Theorem~\ref{thm: Borow1} is tight, Borowiecki, Jendrol, Kr{\'a}l, and Mi{\v s}kuf also proved the following.

\begin{thm}[\cite{BJ06}] \label{thm: Borow2}
Suppose $G$ is a graph with $n$ vertices.  Then, $\chi_\ell(G \square K_{a,b}) = \chi_\ell(G) + a$ whenever $b \geq (\chi_\ell(G) + a - 1)^{an}$.
\end{thm}

One motivation for this paper was: For any given graph $G$, can we improve upon the bound on $b$ in Theorem~\ref{thm: Borow2}?  With this question in mind, for $a \in \N$, we let $f_a(G)$ be the smallest $b$ such that $\chi_\ell(G \square K_{a,b}) = \chi_\ell(G) + a$.  Along the lines of Theorem~\ref{thm: listbipartite}, computing $f_a(G)$ for some graph $G$ and $a \in \N$ yields: $\chi_\ell(G \square K_{a,b}) = \chi_\ell(G) + a$ if and only if $b \geq f_a(G)$. Theorem~\ref{thm: listbipartite} says that $f_a(K_1)= a^a$.

There are several other observations about $f_a(G)$ that are immediate.  First, $\chi_\ell(G \square K_{a,0}) = \chi_\ell(G) < \chi_\ell(G) + a$ which implies that $f_a(G) \geq 1$.  Second, Theorem~\ref{thm: Borow2} implies that $f_a(G) \leq (\chi_\ell(G) + a - 1)^{a|V(G)|}$.  This means $f_a(G)$ exists and is a natural number.  Also, if $G$ is a disconnected graph with components: $H_1, H_2, \ldots, H_r$, we have $f_a(G) = \max_{H_i, \chi_\ell(H_i)= \chi_\ell(G)} f_a(H_i)$.  So, we will restrict our attention to connected graphs.

In this paper we use a list analogue of the chromatic polynomial called the list color function to find an upper bound on $f_a(G)$ that is an improvement on the result of Theorem~\ref{thm: Borow2}.  We also present some further results on computing $f_a(G)$ in the special case where $G$ is a strongly chromatic-choosable graph.

\subsection{The List Color Function and Strong Chromatic-Choosability}

We let $P(G,k)$ be the \emph{chromatic polynomial} of the graph $G$; that is, $P(G,k)$ is equal to the number of proper $k$-colorings of $G$.  It can be shown that $P(G,k)$ is a polynomial in $k$ (see~\cite{B12}).  This notion was extended to list coloring as follows. If $L$ is a list assignment for $G$, we use $P(G,L)$ to denote the number of proper $L$-colorings of $G$. The \emph{list color function} $P_\ell(G,k)$ is the minimum value of $P(G,L)$ where the minimum is taken over all possible $k$-assignments $L$ for $G$.  Since a $k$-assignment could assign the same $k$ colors to every vertex in a graph, it is clear that $P_\ell(G,k) \leq P(G,k)$ for each $k \in \N$.  In general, the list color function can differ significantly from the chromatic polynomial for small values of $k$.  However, Wang, Qian, and Yan~\cite{WQ17} recently showed: If $G$ is a connected graph with $m$ edges, then $P_{\ell}(G,k)=P(G,k)$ whenever $k > \frac{m-1}{\ln(1+ \sqrt{2})}$.  Also see~\cite{AS90} and~\cite{T09} for earlier results on the list color function.

In the case $G$ is a complete graph or odd cycle, it is well known (see~\cite{R68}) that $P(C_{n},k)=(k-1)^{n}+(-1)^n(k-1)$ and $P(K_n,k) = \prod_{i=0}^{n-1} (k-i)$.  It is easy to see that for each $n,k \in \N$, $P(K_n,k)=P_{\ell}(K_n,k)$, and it was recently shown in~\cite{KN16} that for each $n,k \in \N$, $P(C_n,k) = P_{\ell}(C_n,k)$.

In~\cite{KM18} we introduced the notion of strong chromatic-choosability, and we used the list color function to exactly compute $f_1$ for graphs that are strongly chromatic-choosable.  Strong chromatic-choosability is a notion of criticality in the context of chromatic-choosability.  A graph $G$ is \emph{strong k-chromatic-choosable} if $\chi(G) = k$ and every $(k-1)$-assignment, $L$, for which $G$ is not $L$-colorable has the property that the lists are the same on all vertices \footnote{List assignments that assign the same list of colors to every vertex of a graph are called \emph{constant}.}.  We say $G$ is \emph{strongly chromatic-choosable} if it is strong $\chi(G)$-chromatic-choosable. Note that if $G$ is strong $k$-chromatic-choosable, then the only reason $G$ is not $(k-1)$-choosable is that a proper $(k-1)$-coloring of $G$ does not exist. Simple examples of strongly chromatic-choosable graphs include complete graphs, odd cycles, and the join of a complete graph and odd cycle (see~\cite{KM18} for many other examples). See Section~\ref{scc} for a summary of their properties, etc. The following is proven in~\cite{KM18}.

\begin{thm} [\cite{KM18}] \label{thm: star}
Let $M$ be a strong $k$-chromatic-choosable graph.  Then, $f_1(M) = P_\ell(M,k)$.
\end{thm}

We will generalize this result in Theorem~\ref{thm: sccexact}, as stated in the next section.

\subsection{Outline of Results and an Open Questions}

In this subsection we present an outline of the paper while stating our results and mentioning some open questions. Recall that  $f_a(G)$ is defined to be the smallest $b$ such that $\chi_\ell(G \square K_{a,b}) = \chi_\ell(G) + a$.

In Section~\ref{general}, we prove that for any graph $G$, $\chi_\ell(G \square K_{a,b}) = \chi_\ell (G) + a$ whenever $b \geq (P_\ell(G, \chi_\ell(G) + a - 1))^a$. Thus, starting with any chromatic-choosable $G$, and taking its Cartesian product with a sequence of appropriate $K_{a,b}$ (with $a = 0,1,2, \ldots$), we can construct a sequence of graphs that at each step get one farther from being chromatic-choosable: for any $s \ge t \ge 2$ there exists a graph $H$ with $\chi(H)=t$ and $\chi_\ell(H)=s$.

\begin{thm} \label{thm: generalupper}
For any graph $G$ and $a \in \N$, $f_a(G) \leq (P_\ell(G, \chi_\ell(G) + a - 1))^a$.
\end{thm}

It is easy to see that if $G$ has at least one edge, then $P_\ell(G, \chi_\ell(G) + a - 1) < (\chi_\ell(G) + a - 1)^{|V(G)|}$.  This implies that Theorem~\ref{thm: generalupper} is an improvement on Theorem~\ref{thm: Borow2} whenever $G$ has an edge.  We will see many examples in Section~\ref{scc} that illustrate that the bound in Theorem~\ref{thm: generalupper} is tight (notice Theorem~\ref{thm: star} shows the bound is tight when $a=1$ and $G$ is strongly chromatic-choosable).  However, it is not the case that $f_a(G) = (P_\ell(G, \chi_\ell(G) + a - 1))^a$ for all graphs $G$ and $a \in \N$ since it is easy to see that $f_1(C_{2n})=1$, yet $P_\ell(C_{2n}, 2)=2$.  This observation leads us to the following open question.

\begin{ques} \label{ques: uppertight}
For what graphs does $f_a(G) = (P_\ell(G, \chi_\ell(G) + a - 1))^a$ for each $a \in \N$?
\end{ques}

It is possible to slightly modify the proof idea of Theorem~\ref{thm: generalupper} to obtain the following more general result.

\begin{thm} \label{thm: mostgeneral}
Suppose $H$ is a bipartite graph with partite sets $A$ and $B$ where $|A|=a$ and $|B|=b$.  Let $\delta = \min_{v \in B} d_H(v)$.  If $b \geq (P_\ell (G, \chi_\ell(G) + \delta - 1))^a$, then $\chi_\ell (G \square H) \geq \chi_\ell (G) + \delta$.
\end{thm}

Notice that Theorem~\ref{thm: mostgeneral} gives us conditions on when $\chi_\ell(G \square H)$ is guaranteed to be far from $\chi(G \square H)=\max \{\chi(G),2\}$ for any bipartite graph $H$. Furthermore, notice that when $\delta = a$, $H = K_{a,b}$.  So, Theorem~\ref{thm: mostgeneral} implies Theorem~\ref{thm: generalupper}.

In Section~\ref{scc} we prove that if $M$ is a strong $k$-chromatic choosable graph and $k \geq a+1$, then $\chi_\ell(M \square K_{a,b}) = \chi_\ell (G) + a$ if and only if $b \geq (P_\ell(M, \chi_\ell(M) + a - 1))^a$. This is a generalization of Theorem~\ref{thm: star}.

\begin{thm} \label{thm: sccexact}
If $M$ is strongly chromatic-choosable and $\chi(M) \geq a + 1$, then $f_a(M) = (P_\ell(M, \chi_\ell(M) + a - 1))^a.$
\end{thm}

So, we have that the bound in Theorem~\ref{thm: generalupper} is tight when our graph is strongly chromatic-choosable and its chromatic number is sufficiently large.  There are infinite families of such graphs constructed in~\cite{KM18}. Theorem~\ref{thm: sccexact} and the fact that $P(K_n,k)=P_{\ell}(K_n,k)$ and $P(C_n,k) = P_{\ell}(C_n,k)$ whenever $n,k \in \N$ imply the following.

\begin{cor} \label{cor: sccexact}
The following statements hold.
\\
(i)  For any $l \in \N$, $f_2(C_{2l+1})= (P_\ell(C_{2l+1},4))^2 = (3^{2l+1}-3)^2 = 9(9^l-1)^2$.
\\
(ii) For $n \in \N$ satisfying $n \geq a+1$, $f_a(K_n)= (P_\ell(K_n,n+a-1))^a =   \left( \frac{(n+a-1)!}{(a-1)!} \right)^a$.
\end{cor}
Notice that Corollary~\ref{cor: sccexact}~(ii) shows that the bound in Theorem~\ref{thm: generalupper} is tight for any $a \in \N$.

We do not know of any strongly chromatic-choosable graph $M$ for which $f_a(M) < (P_\ell(M, \chi_\ell(M) + a - 1))^a.$  This leads us to the following question.

\begin{ques} \label{ques: scc}
Does there exist a strongly chromatic-choosable graph $M$ such that $f_a(M) < (P_\ell(M, \chi_\ell(M) + a - 1))^a$?
\end{ques}

Another interesting question involves only complete graphs.

\begin{ques} \label{ques: complete}
Is it the case that $f_a(K_n) =   \left( \frac{(n+a-1)!}{(a-1)!} \right)^a$ for each $n, a \in \N$?
\end{ques}

Since $f_a(K_1)=a^a$, the answer to Question~\ref{ques: complete} is yes when $n=1$.  We have a rather tedious argument, which for the sake of brevity will not be presented in this paper, that shows $f_2(K_2)=36$.  One could ask questions analogous to Question~\ref{ques: complete} for any family of strongly chromatic-choosable graphs.  We end Section~\ref{scc} by proving a general lower bound on $f_a$ for strongly chromatic-choosable graphs.

\begin{thm} \label{thm: scclower}
Suppose $M$ is a strong $k$-chromatic-choosable graph.  Then,
$$ \frac{(P_\ell(M,k+a-1))^a}{2^{k-1}} \leq f_a(M).$$
\end{thm}

Considering Theorem~\ref{thm: sccexact}, Theorem~\ref{thm: scclower} gives us something new when $\chi(M) < a+1$.

\section{General Upper Bound} \label{general}

In this section we will prove Theorem~\ref{thm: generalupper}.  Before we prove this theorem, we introduce some notation and terminology that will be used for the remainder of this paper.  Whenever we have a graph of the form $H = G \square K_{a,b}$ with $a,b \in \N$, we will assume that the vertex set of the first factor is $\{v_1, v_2, \ldots, v_n \}$.  We also assume that the partite sets of the copy of $K_{a,b}$ used to form $H$ are $\{u_1, u_2, \ldots, u_a \}$ and $\{w_1, w_2, \ldots, w_b \}$.  If $L$ is a list assignment for $H = G \square K_{a,b}$ and $f$ is a proper $L$-coloring of $H[ \bigcup_{j=1}^a V_{u_j} ]$, then we say $f$ is a \emph{ bad coloring for the copy of $G$ corresponding to $w_l$ } if there is no proper $L'$-coloring for $H[V_{w_l}]$ where $L'$ is the list assignment for $H[V_{w_l}]$ given by $L'(v_i,w_l) = L(v_i,w_l) - \{f(v_i, u_j) : j \in [a] \}$ for each $i \in [n]$ \footnote{We use $L(v_i,w_l)$ rather than the technically correct $L((v_i,w_l))$.}. We now present a straightforward lemma related to this notion of bad coloring.

\begin{lem} \label{lem: badcolor}
Suppose $H = G \square K_{a,b}$ with $a,b \in \N$ and $L$ is a list assignment for $H$.  Suppose $\mathcal{C}$ is the set of all proper $L$-colorings of $H[ \bigcup_{j=1}^a V_{u_j} ]$.  For each $f \in \mathcal{C}$ there exists an $l \in [b]$ such that $f$ is a bad coloring for the copy of $G$ corresponding to $w_l$ if and only if there is no proper $L$-coloring of $H$.
\end{lem}

\begin{proof}
We prove the only if direction first.  Suppose for the sake of contradiction that $c$ is a proper $L$-coloring of $H$.  Let $c'$ be the proper $L$-coloring of $H[ \bigcup_{j=1}^a V_{u_j} ]$ obtained by restricting the domain of $c$ to $\bigcup_{j=1}^a V_{u_j}$ (clearly $c' \in \mathcal{C}$).  We know that there is a $d \in [b]$ such that $c'$ is a bad coloring for the copy of $G$ corresponding to $w_d$.  Let $L'(v_i,w_d) = L(v_i,w_d) - \{c'(v_i, u_j) : j \in [a] \}$ for each $i \in [n]$.  Restricting the domain of $c$ to $V_{w_d}$ yields a proper $L'$-coloring of the copy of $G$ corresponding to $w_d$ which is a contradiction.

We now prove the contrapositive of the converse.  Suppose there is a $g \in \mathcal{C}$ such that for each $t \in [b]$, $g$ is not a bad coloring for the copy of $G$ corresponding to $w_t$.  Let $L'(v_i,w_t) = L(v_i,w_t) - \{c'(v_i, u_j) : j \in [a] \}$ for each $i \in [n]$ and $t \in [b]$.  For each $t \in [b]$, since $g$ is not a bad coloring for the copy of $G$ corresponding to $w_t$, there is a proper $L'$-coloring of $H[V_{w_t}]$.  Coloring each copy of $G$ corresponding to a vertex in $\{w_1, w_2, \ldots, w_b \}$ according to a proper $L'$-coloring extends $g$ to a proper $L$-coloring of $H$.
\end{proof}

We are now ready to prove Theorem~\ref{thm: generalupper}

\begin{proof}
Suppose $H= G \square K_{a,b}$ and $\chi_\ell(G)=k$.  Let $t = P_\ell(G,k+a-1)$.  In order to prove the desired, we must show that $\chi_\ell(H) = k+a$ when $b=t^a$.  We already know $\chi_\ell(H) \leq k + a$ (by Theorem~\ref{thm: Borow1}).  So, we suppose that $b = t^a$, and we will construct a $(k+a-1)$-assignment, $L$, for $H$ such that there is no proper $L$-coloring of $H$.

Let $G_i$ be the copy of $G$ corresponding to $u_i$.  We inductively assign lists of size $(k+a-1)$ to the vertices in $\bigcup_{i=1}^a V_{u_i}$ as follows.  We begin by assigning lists, $L(v)$, to each $v \in V(G_1)$ such that there are exactly $t$ distinct proper $L$-colorings of $G_1$.  Then, for each $1 < i \leq a$, we assign lists, $L(v)$, to each $v \in V(G_i)$ such that there are exactly $t$ distinct proper $L$-colorings of $G_i$ and
$$ \left(\bigcup_{v \in V(G_i)} L(v) \right) \bigcap \left( \bigcup_{j=1}^{i-1} \bigcup_{v \in V(G_j)} L(v) \right ) = \emptyset$$
(this can be done by taking the lists for $G_1$, thinking of the colors as natural numbers, and adding a sufficiently large natural number to each color in each list).  Now, for $i \in [a]$, we let $c_{i,1}, c_{i,2}, \ldots, c_{i,t}$ denote the $t$ distinct proper $L$-colorings of $G_i$.  We note that there are exactly $t^a$ proper $L$-colorings of $H[\bigcup_{i=1}^a V_{u_i}]$ (since for each $i \in [a]$ we have $t$ choices in how we color $G_i$).  Suppose we index the $t^a$ proper $L$-colorings of $H[\bigcup_{i=1}^a V_{u_i}]$ as: $c^{(1)}, c^{(2)}, \ldots, c^{(t^a)}$.

Now, suppose that $L'$ is a $(k-1)$-assignment for $G$ such that there is no proper $L'$-coloring of $G$ and
$$\left(\bigcup_{v \in V(G)} L'(v) \right) \bigcap \left( \bigcup_{j=1}^{a} \bigcup_{v \in V(G_j)} L(v) \right ) = \emptyset.$$
Let $G'_d$ be the copy of $G$ corresponding to $w_d$.  For each $d \in [t^a]$ we assign a list, $L(v)$, of size $(k+a-1)$ to each $v \in V(G'_d)$ as follows.  Suppose that the coloring $c^{(d)}$ is formed via the colorings: $c_{1,b_1}, c_{2,b_2}, \ldots, c_{a,b_a}$ (note that $b_j$ is between 1 and $t$ for each $j \in [a]$).  By construction, we know that $|\{c_{j,b_j}(v_i,u_j) : j \in [a] \}|=a$ for each $i \in [n]$.  So, for each $(v_i,w_d) \in V(G'_d)$, we let
$$L(v_i,w_d) = L'(v_i) \bigcup \{c_{j,b_j}(v_i,u_j) : j \in [a] \}.$$
This completes the construction of our $(k+a-1)$-assignment for $H$.

Finally, notice that by construction $c^{(d)}$ is a bad coloring for $G'_d$ for each $d \in [t^a]$.  Lemma~\ref{lem: badcolor} implies that there is no proper $L$-coloring of $H$.
\end{proof}

It is fairly easy to modify the idea of the proof of Theorem~\ref{thm: generalupper} in order to obtain a proof of Theorem~\ref{thm: mostgeneral}.  We would simply obtain graph $H'$ from $H$ by deleting edges in $H$ until all vertices in $B$ have degree $\delta$.  Then, we could use the same construction idea to obtain a $(\chi_\ell(G) + \delta - 1)$-assignment for $G \square H'$, $L$, such that there is no proper $L$-coloring of $G \square H'$.  This would then imply $\chi_\ell(G) + \delta - 1 < \chi_\ell(G \square H') \leq \chi_\ell (G \square H).$

\section{Computing $f_a$ for Strongly Chromatic-Choosable Graphs} \label{scc}

In this section we will prove Theorems~\ref{thm: sccexact} and~\ref{thm: scclower}.  Suppose that $M$ is a strong $k$-chromatic-choosable graph.  Recall that this means $\chi(M) = k$ and every $(k-1)$-assignment, $L$, for which $M$ is not $L$-colorable is constant.  In this section our focus is studying $f_a(M)$.  By Theorem~\ref{thm: star} we know that $f_1(M) = P_\ell(M,k)$.  So, we assume $a \geq 2$ throughout this section.  Also, since the strong 1-chromatic-choosable graphs are edgeless graphs, we only consider strongly chromatic-choosable graphs with chromatic number at least 2 throughout this section (since we know $f_a(K_1)=a^a$ for each $a \in \N$).

There are several properties of strongly chromatic-choosable graphs that follow immediately from the definition.  We now mention some of these results (all proofs of these results can be found in~\cite{KM18}).

\begin{pro} [\cite{KM18}] \label{pro: obvious} Suppose $M$ is a strong $k$-chromatic-choosable graph.  Then,
\\
(i)  $\chi_{\ell}(M)=k$ (i.e. $M$ is chromatic-choosable); \\
(ii)  If $L$ is a list assignment for $M$ with $|L(v)| \geq k-1$ for each $v \in V(M)$ and $L$ is not a constant $(k-1)$-assignment, then there exists a proper $L$-coloring of $M$; \\
(iii)  $M \vee K_p$ is strong $(k+p)$-chromatic-choosable for any $p \in \N$; \\
(iv)  For any $v \in V(M)$, $\chi(M-\{v\}) \leq \chi_{\ell}(M - \{v\})<k$; \\
(v)  $k=2$ if and only if $M$ is a $K_2$; \\
(vi)  $k=3$ if and only if $M$ is an odd cycle.
\end{pro}

\begin{pro} [\cite{KM18}] \label{pro: stronglistcolor}
Suppose $M$ is a strong $k$-chromatic-choosable graph.  Suppose $L$ is an arbitrary $m$-assignment for $M$ with $m \geq k$.  Then, for any $v \in V(M)$ and any $\alpha \in L(v)$, there is a proper $L$-coloring, $c$, for $M$ such that $c(v) = \alpha$.  Consequently,
$$P_\ell(M,m) \geq m \max_{v \in V(M)} P_\ell(M-\{v\}, m-1) \geq m.$$
\end{pro}

We now prove two lemmas that will lead to the proof of Theorem~\ref{thm: sccexact}.

\begin{lem} \label{lem: disjointlists}
Suppose $M$ is a strong $k$-chromatic-choosable graph and $H= M \square K_{a,b}$.  Suppose that $L$ is a $(k+a-1)$-assignment for $H$ such that there exist $l, x,$ and $y$ with $x \neq y$ and $L(v_l, u_x) \cap L(v_l, u_y) \neq \emptyset$.  Then, there is a proper $L$-coloring of $H$.
\end{lem}

\begin{proof}
Suppose that $\alpha \in L(v_l, u_x) \cap L(v_l, u_y)$.  We begin by finding a proper $L$-coloring for each of the copies of $M$ corresponding to $u_1, \ldots, u_a$.  When it comes to the copies of $M$ corresponding to $u_x$ and $u_y$, we ensure that our proper $L$-colorings for these copies color both $(v_l, u_x)$ and $(v_l, u_y)$ with $\alpha$.  We know this is possible by Proposition~\ref{pro: stronglistcolor}.  For the remaining copies of $M$, we simply take any proper $L$-coloring, and we know there must be at least one such coloring since $k+a-1 \geq k+1 > k$.  Let $c_j$ be the proper $L$-coloring that we found for the copy of $M$ corresponding to $u_j$ for each $j \in [a]$.

Now, for $i \in [n]$ and $t \in [b]$ (we are defining a list assignment for the yet to be colored vertices), we let $L'(v_i, w_t) = L(v_i, w_t) - \{c_j(v_i, u_j) : j \in [a] \}.$  It is easy to see that for each $i$ and $t$, $|L'(v_i,w_t)| \geq k+a-1-a = k-1$, and $|L'(v_l,w_t)| \geq k+a-1 - (a-1) = k$.  So, for any $t \in [b]$, we see that $L'$ restricted to the copy of $M$ corresponding to $w_t$ is not a constant $(k-1)$-assignment.  Proposition~\ref{pro: obvious} then implies that our proper $L$-coloring of $H[ \bigcup_{j=1}^a V_{u_j} ]$ is not bad for the copy of $M$ corresponding to $w_t$.  Lemma~\ref{lem: badcolor} then implies there is a proper $L$-coloring of $H$.
\end{proof}

\begin{lem} \label{lem: badscc}
Suppose $M$ is a strong $k$-chromatic-choosable graph and $H= M \square K_{a,1}$ with $k \geq a+1$.  Suppose that $L$ is a $(k+a-1)$-assignment for $H$ such that the lists \\ $L(v_i,u_1), L(v_i,u_2), \ldots, L(v_i,u_a)$ are pairwise disjoint for each $i \in [n]$.  Then, there is at most one proper $L$-coloring of $H[ \bigcup_{j=1}^a V_{u_j}]$ that is bad for the copy of $M$ corresponding to $w_1$.
\end{lem}

\begin{proof}
Throughout this proof for each $j \in [a]$, we let $M_j$ be the copy of $M$ corresponding to $u_j$.  Furthermore, let $M^*$ be the copy of $M$ corresponding to $w_1$.  Suppose for the sake of contradiction that there exist two distinct proper $L$-colorings, $c$ and $c'$, of $H[ \bigcup_{j=1}^a V_{u_i}]$ that are bad for $M^*$.  Since $M$ is strong $k$-chromatic-choosable, the list assignments:
$$L'(v_i,w_1) = L(v_i,w_1) - \{c(v_i,u_j) : j \in [a] \} \; \; \text{and} \; \; L''(v_i,w_1) = L(v_i,w_1) - \{c'(v_i,u_j) : j \in [a] \}$$
for $i \in [n]$ are both constant $(k-1)$-assignments for $M^*$.  Now, for $j \in [a]$, let $c_j$ be the proper $L$-coloring for $M_j$ obtained by restricting the domain of $c$ to $V(M_j)$, and let $c_{a+j}$ be the proper $L$-coloring for $M_j$ obtained by restricting the domain of $c'$ to $V(M_j)$.  Since $c$ and $c'$ are different, we may assume without loss of generality that $c_1(v_m,u_1) \neq c_{a+1}(v_m,u_1)$ for some $m \in [n]$.  Suppose $c_{a+1}(v_m,u_1)=b$.  Since $L(v_m,u_1), L(v_m,u_2), \ldots, L(v_m,u_a)$ are pairwise disjoint, we know that:
$$ b \notin \bigcup_{j=2}^a L(v_m,u_j) \; \; \text{and} \; \; b \notin \{c_j(v_m,u_j) : j \in [a] \}.$$
Let $A$ be the set of $(k-1)$ colors that $L'$ assigns to all the vertices in $M^*$, and let $B$ be the set of $(k-1)$ colors that $L''$ assigns to all the vertices in $M^*$.  We know that for $i \in [n]$, $L(v_i,w_1) = A \bigcup \{c_j(v_i,u_j) : j \in [a] \}.$ So, for $i \in [n]$,
\begin{equation} \label{eq: 1}
B = \left (A \bigcup \{c_j(v_i,u_j) : j \in [a] \} \right ) - \{c_{a+j}(v_i,u_j) : j \in [a] \}.
\end{equation}
Since $|B|=k-1$, $b \notin \{c_j(v_m,u_j) : j \in [a] \}$, and equation~(\ref{eq: 1}) holds for $i=m$,  it must be the case that $b \in A$.  It immediately follows that $b \in \{c_{a+j}(v_i,u_j) : j \in [a] \}$ for each $i \in [n]$.  To see why, note that if this did not hold we would have that some of the lists obtained from $L''$ would contain $b$ and some would not.

Now, for $j \in [a]$, let
$$C_j = \{ v_l : c_{a+j}(v_l,u_j) = b \}.$$
Note $C_1, \ldots, C_a$ are pairwise disjoint because if $C_r$ and $C_s$ both contained some vertex $v_p$ and $r \neq s$, then $b \in L(v_p,u_r) \cap L(v_p, u_s)$ which is a contradiction.  Since $b \in \{c_{a+j}(v_i,u_j) : j \in [a] \}$ for each $i \in [n]$, we have that $V(M) = \bigcup_{j=1}^a C_j$.  Since $c_{a+j}$ colors all the vertices in $V(M_j)$ with first coordinate in $C_j$ with the color $b$, we know that $C_j$ is an independent set of vertices in $M$.  Thus, $\{C_1, \ldots, C_a \}$ is a partition of $V(M)$ into $a$ independent sets.  This implies that $k=\chi(M) \leq a$ which is a contradiction.  This completes the proof.
\end{proof}

We now restate and prove Theorem~\ref{thm: sccexact}.

\begin{customthm}{\bf \ref{thm: sccexact}}
If $M$ is strongly chromatic-choosable and $\chi(M) \geq a + 1$, then $f_a(M) = (P_\ell(M, \chi_\ell(M) + a - 1))^a.$
\end{customthm}

\begin{proof}
We know by Theorem~\ref{thm: generalupper}, $f_a(M) \leq (P_\ell(M, \chi_\ell(M) + a - 1))^a$.  Suppose that $M$ is strong $k$-chromatic-choosable with $k \geq a+1$ and $H= M \square K_{a,b}$.  To prove the desired, we must show that if $b < (P_\ell(M,k+a-1))^a$, then $\chi_\ell(H) < k+a$.

Let $t = P_\ell(M,k+a-1)$, $M_i$ be the copy of $M$ corresponding to $u_i$, and $M'_j$ be the copy of $M$ corresponding to $w_j$.  We assume that $b < t^a$, and we let $L$ be an arbitrary $(k+a-1)$-assignment for $H$.  To prove the desired, we will show that there is a proper $L$-coloring of $H$.  By Lemma~\ref{lem: disjointlists}, we may assume the lists $L(v_i,u_1), L(v_i,u_2), \ldots, L(v_i,u_a)$ are pairwise disjoint for each $i \in [n]$.  For each $j \in [a]$, there are clearly at least $t$ distinct proper $L$-colorings of $M_j$.  This implies that there are at least $t^a$ proper $L$-colorings of $H[\bigcup_{j=1}^a V_{u_j}]$.  Let $\mathcal{C}$ be the set of distinct proper $L$-colorings of $H[\bigcup_{j=1}^a V_{u_j}]$ (we know $|\mathcal{C}| \geq t^a$).

By Lemma~\ref{lem: badscc} we know that for each $d \in [b]$, there is at most one coloring in $\mathcal{C}$ that is bad for $M'_d$. Since $b < t^a \leq |\mathcal{C}|$, there must be some $f \in \mathcal{C}$ that is not a bad coloring for any of: $M'_1, \ldots, M'_b$.  Lemma~\ref{lem: badcolor} then implies there is a proper $L$-coloring of $H$.
\end{proof}

Theorem~\ref{thm: sccexact} along with the fact that the list color function of an odd cycle (resp. complete graph) is equal to the chromatic polynomial of the odd cycle (resp. complete graph) for each natural number immediately yields the following corollary.

\begin{customcor}{\bf \ref{cor: sccexact}}
The following statements hold.
\\
(i)  For any $l \in \N$, $f_2(C_{2l+1})= (P_\ell(C_{2l+1},4))^2 = (3^{2l+1}-3)^2 = 9(9^l-1)^2$.
\\
(ii) For $n \in \N$ satisfying $n \geq a+1$, $f_a(K_n)= (P_\ell(K_n,n+a-1))^a =   \left( \frac{(n+a-1)!}{(a-1)!} \right)^a$.
\end{customcor}

Notice that Corollary~\ref{cor: sccexact} Statement~(ii) shows that the bound in Theorem~\ref{thm: generalupper} is tight for any $a \in \N$.  Suppose $M$ is a strongly chromatic-choosable graph.  At this stage of the paper, we have $f_a(M)$ exactly in terms of the list color function of $M$ when $\chi(M) \geq a+1$.  When $\chi(M) < a+1$ we only have an upper bound on $f_a(M)$ in terms of the list color function of $M$ (by Theorem~\ref{thm: generalupper}).  We will now turn our attention to proving Theorem~\ref{thm: scclower} which will give us a lower bound on $f_a(M)$ in terms of the list color function of $M$ when $\chi(M) < a+1$.  We begin with a lemma.

\begin{lem} \label{lem: scclower}
Suppose $M$ is strong $k$-chromatic-choosable and $H= M \square K_{a, 1}$ where $k < a+1$.  Let $L$ be a $(k+a-1)$-assignment for $H$ such that the lists $L(v_i,u_1), L(v_i,u_2), \ldots, L(v_i,u_a)$ are pairwise disjoint for each $i \in [n]$.  Let $\mathcal{B}$ be the set of proper $L$-colorings of $H[ \bigcup_{j=1}^a V_{u_j}]$ that are bad for the copy of $M$ corresponding to $w_1$.  Then, $|\mathcal{B}| \leq 2^{k-1}$.
\end{lem}

\begin{proof}
We are done if $\mathcal{B}$ is empty.  So we suppose $\mathcal{B}$ has at least one element.  For $j \in [a]$, let $X_j = \{ c : c \in L(v_1, w_1) \cap L(v_1, u_j) \}$ and $x_j = |X_j|$.  Note that $X_1, X_2, \ldots, X_a$ must be pairwise disjoint.  We claim that
$$ |\mathcal{B}| \leq \left |\prod_{j=1}^a X_j \right | = \prod_{j=1}^a x_j$$
where $\prod_{j=1}^a X_j$ is the Cartesian product of the sets: $X_1, X_2, \ldots, X_a$.  Now, suppose that $f \in \mathcal{B}$.  Then, for $i \in [n]$, we know that the list assignment $L'$ given by
$$L'(v_i,w_1) = L(v_i,w_1) - \{ f(v_i, u_j) : j \in [a] \}$$
is a constant $(k-1)$-assignment for the copy of $M$ corresponding to $w_1$.  This implies that $\{f(v_1, u_1), f(v_1, u_2), \ldots, f(v_1, u_a) \}$ is a set of size $a$ that is completely contained in $L(v_1, w_1)$.   Thus,
$$(f(v_1, u_1), f(v_1, u_2), \ldots, f(v_1, u_a)) \in \prod_{j=1}^a X_j.$$
So, if for each $f \in \mathcal{B}$, we let $T(f) = (f(v_1, u_1), f(v_1, u_2), \ldots, f(v_1, u_a))$, we see that $T$ is a function from $\mathcal{B}$ to $\prod_{j=1}^a X_j$.  In order to prove the desired, we will show that $T$ is injective.  For the sake of contradiction, suppose $f$ and $g$ are distinct colorings in $\mathcal{B}$ such that $T(f) = T(g)$.  For each $i \in [n]$ let $L'$ and $L''$ be the list assignments for the copy of $M$ corresponding to $w_1$ given by
\begin{align*}
&L'(v_i,w_1) = L(v_i,w_1) - \{ f(v_i, u_j) : j \in [a] \} \; \; \text{and} \\
&L''(v_i,w_1) = L(v_i,w_1) - \{ g(v_i, u_j) : j \in [a] \}.
\end{align*}
We know that $L'$ and $L''$ are both constant $(k-1)$ assignments.  Since $T(f)=T(g)$, $L'(v_1,w_1) = L''(v_1, w_1)$ which immediately implies $L'$ and $L''$ assign some list, $A$, of $(k-1)$ colors to every vertex in the copy of $M$ corresponding to $w_1$.  Since $f \neq g$ there are constants, $r$ and $t$, such that $f(v_r, u_t) \neq g(v_r, u_t)$.  Since $L'$ and $L''$ are constant $(k-1)$ assignments assigning $A$ to every vertex in the copy of $M$ corresponding to $w_1$, we know that $\{ f(v_r, u_j) : j \in [a] \} = \{ g(v_r, u_j) : j \in [a] \}$.  Since these two sets must have size $a$, $f(v_r, u_t) = g(v_r, u_p)$ for some $p \neq t$.  This however contradicts the fact that $L(v_r, u_t)$ and $L(v_r, u_p)$ are disjoint.  Thus, $T$ is injective, and we have $|\mathcal{B}| \leq \left |\prod_{j=1}^a X_j \right |$.

To finish the proof, we must show that $\prod_{j=1}^a x_j \leq 2^{k-1}$.  Notice that each $x_j$ is a nonnegative integer such that $ \sum_{j=1}^a x_j \leq k+a-1$.  Under these conditions, the maximum possible value of $\prod_{j=1}^a x_j $ is achieved when $ \sum_{j=1}^a x_j = k+a-1$ and each $x_j$ equals $\lfloor (k+a-1)/a \rfloor= 1$ or $\lceil (k+a-1)/a \rceil = 2$; that is, when $(k-1)$ of $x_1, x_2, \ldots, x_a$ are 2 and the rest are 1.
\end{proof}

It is fairly easy to show that the bound in Lemma~\ref{lem: scclower} is tight for each $k \geq 2$.  For example, suppose $G = K_3$, $V(G) = \{v_1, v_2, v_3 \}$, and $H = G \square K_{3,1}$.  Suppose $L$ is the 5-assignment for $H$ that assigns: $\{1,2,3,4,5 \}$ to $(v_i,w_1)$ for each $i \in [3]$, $\{1,2,6,7,8 \}$ to $L(v_1,u_1)$, $L(v_3,u_2)$, and $L(v_2,u_3)$, $\{3,4,9,10,11 \}$ to $L(v_2,u_1)$, $L(v_1,u_2)$, and $L(v_3,u_3)$, and $\{5,12,13,14,15 \}$ to $L(v_3,u_1)$, $L(v_2,u_2)$, and $L(v_1,u_3)$.  It is easy to see that $L$ satisfies the hypotheses of Lemma~\ref{lem: scclower}, and there are exactly 4 proper $L$-colorings of $H[ \bigcup_{j=1}^3 V_{u_j}]$ that are bad for the copy of $G$ corresponding to $w_1$.    We are now ready to restate and prove Theorem~\ref{thm: scclower}.

\begin{customthm}{\bf \ref{thm: scclower}}
Suppose $M$ is a strong $k$-chromatic-choosable graph.  Then,
$$ \frac{(P_\ell(M,k+a-1))^a}{2^{k-1}} \leq f_a(M).$$
\end{customthm}

\begin{proof}
Suppose that $M$ is strong $k$-chromatic-choosable and $H= M \square K_{a,b}$.  By Theorem~\ref{thm: sccexact}, the result is obvious when $k \geq a + 1$.  So, we assume $k < a+1$.  Let $t= P_\ell(M,k+a-1)$, and let $M_i$ be the copy of $M$ in $H$ corresponding to $w_i$.  We assume $b < t^a/2^{k-1}$, and we let $L$ be an arbitrary $(k+a-1)$-assignment for $H$.  To prove the desired, we will show there is proper $L$-coloring for $H$.  By Lemma~\ref{lem: disjointlists}, we may assume that the lists $L(v_i,u_1), L(v_i,u_2), \ldots, L(v_i,u_a)$ are pairwise disjoint for each $i \in [n]$.  Let $\mathcal{C}$ be the set of proper $L$-colorings of $H[ \bigcup_{j=1}^a V_{u_j}]$.  Clearly, $t^a \leq |\mathcal{C}|$.  For $d \in [b]$ let $C_d$ be the subset of $\mathcal{C}$ that contains all the proper $L$-colorings of $H[ \bigcup_{j=1}^a V_{u_j}]$ that are bad for $M_d$.  By Lemma~\ref{lem: scclower} we have that $|C_d| \leq 2^{k-1}$.  Since $b < t^a/2^{k-1}$,
$$ \sum_{d=1}^b |C_d| \leq b(2^{k-1}) < t^a \leq |\mathcal{C}|.$$
Thus, $\bigcup_{d=1}^{b} C_d$ must be a proper subset of $\mathcal{C}$, and we can find an $f \in \mathcal{C} - \bigcup_{d=1}^{b} C_d$.   Lemma~\ref{lem: badcolor} then implies there is a proper $L$-coloring of $H$.
\end{proof}

\end{document}